\documentclass[12pt]{amsart} 
\usepackage{amsthm,amsbsy,amsfonts,amssymb,amscd}

\usepackage[mathscr]{euscript} 

\usepackage[T1]{fontenc} 
\usepackage{sidenotes} 

\usepackage{enumerate} 
\usepackage{tikz} 
\usetikzlibrary{calc,intersections,through, backgrounds,arrows,decorations.pathmorphing,backgrounds,positioning,fit,petri} 
\usetikzlibrary{calc}
 
\usepackage{tikz-cd} 

\usepackage{hyperref}

\usepackage[margin=1in]{geometry}

\numberwithin{equation}{subsection}

\newtheoremstyle{notes} {} {} {} {} {\bfseries} {.} {.5em} {}

\theoremstyle{plain}

\newtheorem{prop}[subsection]{Proposition}

\newtheorem{lemma}[subsection]{Lemma}

\newtheorem{cor}[subsection]{Corollary}

\newtheorem{thm}[subsection]{Theorem}

\theoremstyle{remark}

\newtheorem{rem}[subsection]{Remark} 

\swapnumbers

\theoremstyle{remark}

\newcommand{\tto}{\twoheadrightarrow}

\newcommand{\KL}{\unlhd}

\usepackage{graphicx}
\usepackage[all]{xy}

\newcommand{\Hom}{{\mathrm{Hom}}}
\newcommand{\Ext}{{\mathrm{Ext}}}

\newcommand{\Rad}{{\mathrm{Rad}}}

\newcommand{\eins}{\leavevmode\hbox{\small1\kern-3.8pt\normalsize1}}
\newcommand{\op}{{\mathrm{op}}}
\newcommand{\res}{{\rm Res} }
\newcommand{\Res}{{\rm Res}^{\fg}_{\fg_{\oa}} }
\newcommand{\ind}{{\rm Ind} }
\newcommand{\Ind}{{\rm Ind}^{\fg}_{\fg_{\oa}} }
\newcommand{\Ann}{{\rm Ann} }

\newcommand{\len}{\mathtt{l}}

\newcommand{\mC}{\mathbb{C}}
\newcommand{\mN}{\mathbb{N}}

\newcommand{\mZ}{\mathbb{Z}}

\newcommand{\ad}{{\mathrm{ad}}}

\newcommand{\dd}{{\mathbf{d}}}

\newcommand{\UU}{\mathbf{U}}

\newcommand{\fb}{{\mathfrak b}}

\newcommand{\fg}{{\mathfrak g}}
\newcommand{\fh}{{\mathfrak h}}

\newcommand{\fn}{{\mathfrak n}}

\newcommand{\fm}{{\mathfrak m}}

\newcommand{\cD}{\mathcal{D}}

\newcommand{\cH}{\mathcal{H}}

\newcommand{\cF}{\mathcal{F}}

\newcommand{\cO}{\mathcal{O}}

\newcommand{\cL}{\mathcal{L}}
\newcommand{\cB}{\mathcal{B}}

\newtheorem{pfVogan}[subsection]{Proof of Theorem~\ref{thmVogan}}

\newtheorem{pfAS}[subsection]{Proof of Theorem~\ref{thmAS}}

\newtheorem{pfMain}[subsection]{Proof of Theorem~\ref{thmMain}}

\newtheorem{pfGeneric}[subsection]{Proof of Theorem~\ref{generic}}

\newcommand{\oa}{\bar{0}}
\newcommand{\ob}{\bar{1}}

\newcommand{\End}{{\rm End}}

\newcommand{\Zf}{{\fg\mbox{-mod}_{Z}}}
\newcommand{\Lann}{{\rm LAnn}}

\newtheoremstyle{construction} {} {} {} {} {\bfseries} { } {0pt} {}

\theoremstyle{construction}

\title[Super primitive spectrum]{The primitive spectrum of basic classical Lie superalgebras} 

\author{Kevin~Coulembier}



\begin{document} 
\date{} 
\begin{abstract}
We prove Conjecture 5.7 in [arXiv:1409.2532], describing all inclusions between primitive ideals for the general linear superalgebra in terms of the $\Ext^1$-quiver of simple highest weight modules. For arbitrary basic classical Lie superalgebras,
we formulate two types of Kazhdan-Lusztig quasi-orders on the dual of the Cartan subalgebra, where one corresponds to the above conjecture. Both orders can be seen as generalisations of the left Kazhdan-Lusztig order on Hecke algebras and are related to categorical braid group actions. We prove that the primitive spectrum is always described by one
of the orders, obtaining for the first time a description of the inclusions. We also prove that the two orders are identical if category~$\cO$ admits `enough' abstract Kazhdan-Lusztig theories. In particular, they are identical for the general
linear superalgebra, concluding the proof of the conjecture.
	\end{abstract}

\maketitle

\vspace{-3mm}

\section{Introduction}

\subsection{} We refer to the primitive ideals in the universal enveloping algebra of a Lie (super)algebra, {\it i.e.} the annihilator ideals of the simple modules, as the primitive ideals of the Lie (super)algebra. A principal motivation to study these, stems from the open problem of classifying simple modules. Except for~$\mathfrak{sl}(2)$ and some related Lie superalgebras, a solution to this problem seems to be far out of reach, for the moment. This motivates the study of rougher invariants for simple modules, such as annihilator ideals.

\subsection{}In \cite[Th\'eor\`eme~1]{Duflo}, Duflo proved that, for a reductive {\em Lie algebra}, every primitive ideal is equal to the annihilator ideal of some simple {\em highest weight} module, for an arbitrary but fixed Borel subalgebra. In order to obtain a classification of primitive ideals, it remained to determine for which simple highest weight modules the corresponding annihilator ideals are equal, and more generally, when there is an inclusion. In \cite[Theorem~2.12]{Borho}, Borho and Jantzen reduced this problem to regular central characters, by using a translation principle. The correct description for a regular central character was then conjectured by Joseph in \cite[Conjecture~5.7]{Joseph1} and proved by Vogan in \cite[Theorem~3.2]{Vogan}. The inclusion order for a regular central character is the left Kazhdan-Lusztig (KL) order of~\cite{KL}, which followed {\it a posteriori} from the validity of the KL conjecture, as proved in \cite{BB, Kashiwara, EW}.

\subsection{} For simple classical {\em Lie superalgebras}, the analogue of Duflo's theorem was proved by Musson in \cite{Musson}, leaving again the question for which simple highest weight modules inclusions between annihilator ideals occur. An overview of some complications is given in \cite[Section~1]{CouMus}, while the concrete example in \cite[Section~7]{CouMus} demonstrates that inclusions are far more prevalent compared to the Lie algebra case. Since the result in \cite{Musson}, a number of partial answers and special cases were obtained by several authors. In \cite{sl21, osp12r}, Musson solved the problem for~$\mathfrak{sl}(2|1)$ and~$\mathfrak{osp}(1|2n)$, and in~\cite{q2}, Mazorchuk solved the case~$\mathfrak{q}(2)$. In~\cite{Letzter}, Letzter obtained a classification of the primitive ideals for~$\mathfrak{gl}(m|n)$, but the question of which (strict) inclusions occur was not settled. In \cite{CM1} Mazorchuk and the author obtained a general principle, for classical Lie superalgebras, to derive inclusions through {\em twisting functors}. In the current paper we will prove that this procedure is in fact exhaustive. Furthermore, in \cite{CM1} the equalities between annihilator ideals in the generic region, as defined in Definition~7.1 of {\it op. cit.}, were classified. The description required the introduction of a new deformation of the Weyl group orbits, called the {\em star orbits}.  In \cite{CouMus}, Musson and the author obtained all inclusions for~$\mathfrak{gl}(m|1)$ and formulated for the first time a complete conjectural description for the inclusions for~$\mathfrak{gl}(m|n)$. We prove this conjecture in the current paper. The inclusion order is determined by the knowledge of the first extensions between simple modules in $\cO$ in \cite{CLW, BLW}. 

\subsection{}
In \cite{KL}, Kazhdan and Lusztig defined two quasi-orders (called left and right KL order) on Coxeter groups and their Hecke algebras. The way these orders come into play in the primitive spectrum for Lie algebras originates in the {\em categorification} of the {\em integral group algebra~$\mZ[W]$} of the Weyl group $W$, on the principal block $\cO_0$ of the BGG category~$\cO$, through {\em projective functors} in \cite{BG}. Alternatively, the left and right order can be introduced on~$\cO_0$, by categorifications of the {\em braid group $\cB_W$}, see \cite{AS, shuff, braid}. This is given in terms of {\em twisting functors}, for the left order, or through {\em shuffling functors}, for the right order. We refer to \cite[Sections 5 and 6]{AlgCat} for an overview of these categorifications. In singular blocks, twisting functors still behave well, contrary to projective and shuffling functors. As twisting commutes with translation functors, they also introduce a left KL order on singular blocks, coinciding with the inclusion order. 

Similarly, we define a {\em completed left KL order} for superalgebras in terms of twisting functors, where the latter again correspond to a braid group action, by \cite{CM1}. In \cite{CouMus}, another KL order was defined for superalgebras, as an analogue of the realisation on $\cO$ of the left order for Lie algebras, which is determined by the~$\Ext^1$-quiver. We prove, by studying Jantzen middles, that for superalgebras for which $\cO$ admits suitable abstract KL theories in the sense of~\cite{CPS}, our completed KL order is identical to the KL order of \cite{CouMus}. In particular, this is true for~$\mathfrak{gl}(m|n)$. We also provide an alternative formulation of the KL order in terms of the {\em partial approximation} principle introduced in \cite{Khomenko}. Finally, we show that in the {\em generic region}, the KL order is described by the combination of the left KL order on the Weyl group of \cite{KL} and the deformed Weyl group orbits introduced in \cite{CM1}.

\subsection{} The paper is organised as follows. In Section~\ref{secPrel} we gather the necessary preliminaries. In Section~\ref{secTools} we develop some general principles which will be useful in the remainder. In Section~\ref{secHC} we obtain some results on Harish-Chandra bimodules, similar to \cite{BG, Vogan}. Some of these had already essentially been obtained in \cite{Serre}. In Section~\ref{secCompletion} we introduce Enright's completion functors on category~$\cO$ and relate them to twisting and partial approximation. In Section~\ref{secAS} we study an analogue of the Jantzen middle, which appeared for Lie algebras in \cite{AS}. In Section~\ref{secConclusion} we bring all the above together to prove \cite[Conjecture~5.7]{CouMus} and describe the primitive spectrum for arbitrary basic classical Lie superalgebras. In Section~\ref{secgeneric} we describe the completed KL order explicitly for the generic region.

\subsection{} In the current paper, we do not consider the classical Lie superalgebras of `queer type'. The main reason being that a simultaneous treatment of queer type and basic classical type seems to lead to very intricate formulation of statements. However, the main ingredients of this paper have also been developed in that setting, {\it viz.} strongly typical blocks have been studied in \cite{Frisk}, twisting functors in \cite{CM1} and Harish-Chandra bimodules in \cite{Serre}. It would be particularly interesting to compare the solution for~$\mathfrak{q}(n)$ with the star action of~\cite{Dimitar}, as has already been done for the generic region in \cite{CM1}.

\subsection{} Finally, we note that this paper offers a rather self-contained proof of the ordering of the primitive spectrum for reductive Lie algebras. In particular, the proof bypasses translation to the walls, as developed in \cite{Borho}. We also benefit from the fact that, contrary to when the breakthroughs in \cite{Joseph1, Vogan} were obtained, by now the KL conjecture is a theorem, projective functors have been classified, twisting and completion functors have been introduced and the equivalence between category~$\cO$ and a certain category of Harish-Chandra bimodules has been established. These facts allow for a more elegant treatment.

\section{Preliminaries}\label{secPrel}
We set $\mN=\{0,1,2,\cdots\}$.
We always work over the ground field of complex numbers $\mC$. In particular $\otimes_\mC$, $\Hom_\mC$ and~$\dim_\mC$ will be abbreviated to $\otimes$, $\Hom$ and~$\dim$. For an associative algebra $A$ over $\mC$ we denote the category of finitely generated modules by $A$-mod and the category of all modules by $A$-Mod. We are interested in the Lie superalgebras 
\begin{equation}\label{list}\mathfrak{gl}(m|n),\;\,\mathfrak{sl}(m|n),\;\, \mathfrak{psl}(n|n),\;\,\mathfrak{osp}(m|2n),\;\,D(2,1;\alpha),\;\, G(3)\mbox{ and }\;\, F(4),\end{equation}
see \cite{bookMusson} for a complete treatment of such Lie superalgebras and their universal enveloping algebras. Whenever we consider a Lie superalgebra `$\fg$', it is assumed to be in this list.

\subsection{}\label{Prel111}
As a super vector space,~$\fg$ decomposes into its even and odd subspace as~$\fg=\fg_{\oa}\oplus\fg_{\ob}$,
where~$\fg_{\oa}$ is a reductive Lie algebra. We denote the universal enveloping algebra by $U=U(\fg)$ and its centre by $Z(\fg)$. We also use $U_{\oa}=U(\fg_{\oa})$. As $U$ is a finite extension over the noetherian ring $U_{\oa}$, it is noetherian itself. We consider the antiautomorphism $t$ on~$\fg$ given by $t(X)=-X$ which extends to an antiautomorphism of $U$.

We set~$\fg\mbox{-mod} =U\mbox{-mod}$. The full subcategory of~$\fg$-mod on which~$Z(\fg)$ acts locally finitely is denoted by~$\Zf$. For a central character~$\chi:Z(\fg)\to\mC$, we denote by~$\fg\mbox{-mod}_\chi$ the category of modules annihilated by some power of~$\fm_{\chi}:=\ker\chi$.
The endofunctor $(-)_\chi$ of~$\Zf$ is defined as taking the largest summand in~$\fg\mbox{-mod}_{\chi}$. We have two exact functors$$\res:=\Res:\fg\mbox{-mod}\to\fg_{\oa}\mbox{-mod}\quad\mbox{and}\quad\ind:=\Ind:\fg_{\oa}\mbox{-mod}\to\fg\mbox{-mod},$$
The functor~$\res$ is both left and right adjoint to $\ind$, as induction and coinduction are isomorphic, see e.g. \cite[Section~2.3.5]{Gorelika}. 

\subsection{}\label{secrho} Consider a triangular decomposition $\fg=\fn^-\oplus\fh\oplus \fn^+,$
where~$\fh\subset\fg_{\oa}$ is referred to as the Cartan subalgebra and~$\fb=\fh\oplus\fn^+$ as the Borel subalgebra. 
The set positive roots of~$\fg_{\oa}$ with respect to the above triangular decomposition is denoted by $\Delta^+_{\oa}$ and the odd positive roots by $\Delta_{\ob}^+$. The partial order $\le$ on $\fh^\ast$ is transitively generated by $\lambda-\beta\le\lambda$ for any $\beta\in \Delta^+$.

We denote by $\rho$ the half sum $\rho_{\oa}$ of even positive roots minus the half sum $\rho_{\ob}$ of odd positive roots. The Weyl group of~$\fg_{\oa}$
is denoted by $W=W(\fg_{\oa}:\fh)$. By the term {\em simple even root} we always mean a positive even root, which is simple in~$\Delta_{\oa}^+$, or equivalently corresponds to a simple
reflection in the Coxeter group $W$. Such a root is {\em not necessarily simple in~$\Delta^+$}. We assume a $W$-invariant
form $\langle\cdot,\cdot\rangle$ on~$\fh^\ast$ as in \cite[Theorem~5.4.1]{bookMusson} and set $\alpha^\vee=2\alpha/\langle \alpha,\alpha\rangle$ for any even root $\alpha$. We consider the actions $w\cdot\lambda=w(\lambda+\rho)-\rho$ and~$w\circ\lambda=w(\lambda+\rho_{\oa})-\rho_{\oa}$ for all $w\in W$ and~$\lambda\in \fh^\ast$. 

A weight $\lambda$ is called {\em integral} if $\langle\lambda,\alpha^\vee\rangle\in\mZ$ for all (simple) even roots~$\alpha$. The set of integral weights is denoted by $\Lambda_0\subset\fh^\ast$. Furthermore, for any $\lambda\in\Lambda$ we set $\Lambda_\lambda=\lambda+\Lambda_0$. For arbitrary
$\Lambda\in\fh^\ast/\Lambda_0$ we set $\Delta^+_{\oa,\Lambda}$ equal to the subset of $\Delta^+_{\oa}$ of $\alpha$ for which $\langle \lambda,\alpha^\vee\rangle\in\mZ$ for~$\lambda\in\Lambda$ arbitrary. We define a Coxeter group $W_\Lambda$ generated by the reflections~$s_\alpha\in W$ for which $\alpha$ becomes simple in~$\Delta^+_{\oa,\Lambda}$. For $w\in W_{\Lambda}$ and $\lambda\in \Lambda$ we have~$w\cdot \lambda\in\Lambda$. A weight $\lambda\in\Lambda$
is called {\em dominant} if $\lambda$ is dominant for the $\rho$-shifted~$W_\Lambda$-action and {\em regular} if and only if $\langle \lambda+\rho,\alpha^\vee\rangle\not=0$ for all simple even roots~$\alpha$. 

We say that~$\lambda$ is {\em strongly typical} if $\langle\lambda+\rho,\gamma\rangle\not=0$ for all $\gamma\in\Delta_{\ob}$. This is equivalent to the condition that the central character of the module~$L(\lambda)$ is strongly typical in the sense of~\cite[Section~1.3]{Gorelik}, by \cite[Section~2.3]{Gorelik}. A special role will be played by weights which are simultaneously regular, strongly typical and dominant. Therefore we introduce the term {\em super dominant} weights for these.

\begin{lemma}\label{lemZar}
For an arbitrary $\mu\in \fh^\ast$, there is a $\kappa\in \Lambda_0$ such that~$L(\kappa)$ is finite dimensional and for which $\mu+\kappa$ is regular strongly typical.
\end{lemma}
\begin{proof}
For any integer $d$ set 
$$\Phi_d=\{\kappa\in\fh^\ast\,|\, \langle \kappa+\rho,\alpha^\vee\rangle\,\in\mZ_{\ge d},\; \;\forall\;\alpha\in\Delta_{\oa}^+\}.$$
There exists~$d_0\in\mN$ such that $L(\kappa)$ is finite dimensional if $\kappa\in\Phi_{d_0}$.
This follows for instance from~\cite[Lemma~7.2]{CM1}.
The set $\mu+\Phi_{d_0}$ is Zariski dense in $\fh^\ast$. On the other hand, the set of regular strongly typical weights is Zariski open, as it is the complement of a union of hyperplanes. A dense and an open set always have a non-trivial intersection.
\end{proof}

\subsection{}\label{secO}

We consider the BGG category~$\cO(\fg,\fb)$ associated to
the triangular decomposition of~$\fg$, see \cite{BGG, Humphreys, bookMusson}. Aside from
the classical definition, $\cO(\fg,\fb)$ is the full subcategory of~$\fg$-mod of modules $M$
such that~$\Res M$ is an object in~$\cO(\fg_{\oa},\fb_{\oa})$. Note that~$\ind$ also restricts
to a functor $\cO(\fg_{\oa},\fb_{\oa})\to \cO(\fg,\fb)$.

By the above it is clear that~$\cO(\fg,\fb)$ only depends on~$\fb_{\oa}=\fg_{\oa}\cap\fb$, so we could use the notation~$\cO(\fg,\fb_{\oa})$. However, for each different~$\fb$, with same underlying~$\fb_{\oa}$, the category has a different {\em highest weight category} structure, for $(\fh^\ast,\le)$, with standard modules the Verma modules introduced below. Therefore we will use the notation~$\cO(\fg,\fb)$ to the denote the category~$\cO(\fg,\fb_{\oa})$ with the specific structure of a highest weight category corresponding to~$\fb$. When no confusion is possible, we just write $\cO$. We consider the usual simple-preserving duality~$\dd$ on~$\cO$, see e.g. \cite[Section~3.2]{Humphreys}.

A complete set of non-isomorphic simple objects in category~$\cO$ is given by the simple tops of the
Verma modules $\Delta(\lambda)=U(\fg)_{\otimes U(\fb)}\mC_\lambda$ for all $\lambda\in \fh^\ast$, where $\mC_\lambda$ is a one-dimensional~$\fb$-module either purely even or purely odd. It will not be important for us to make a distinction between both, so we ignore parity. We denote these modules by $L(\lambda)$. We also set $\chi_\lambda$ equal to the central character of $L(\lambda)$.
The indecomposable projective cover of~$L(\lambda)$ in~$\cO$ is denoted by $P(\lambda)$ and we set $\nabla(\lambda)=\dd\Delta(\lambda)$.
To make a distinction with the correspondingly defined modules in~$\cO(\fg_{\oa},\fb_{\oa})$ we will denote the latter by $L_0(\lambda)$,
$\Delta_0(\lambda)$ and $P_0(\lambda)$.

For any central character~$\chi$ we denote by $\cO_\chi$ the full subcategory of modules in~$\cO$ which are in~$\fg\mbox{-mod}_\chi$. Similarly for any $\Lambda\in \fh^\ast/\Lambda_0$ we denote by $\cO^\Lambda$ the full subcategory of modules which have only
non-zero weight spaces for weights in~$\Lambda$. Then we have decompositions
$$\cO\;=\;\bigoplus_{\chi}\cO_\chi\;=\;\bigoplus_{\Lambda}\cO^\Lambda\;=\; \bigoplus_{\chi,\Lambda}\cO^\Lambda_\chi.$$
We also set $\Lambda^\chi$ equal to the set of $\lambda\in\Lambda$ for which $\chi_\lambda=\chi$.

When $\lambda$ is super dominant, we call $P(\lambda)$ a {\em super dominant projective} module. Central character considerations
imply that~$P(\lambda)=\Delta(\lambda)$ for super dominant $\lambda$. A justification for the term super dominant is that the category~$\cO^\Lambda$ has the structure of a category with full projective functors where~$P(\lambda)$ is the `dominant object', according to \cite[Definition~1.1]{PF}. This can be proved as in \cite[Theorem~5.1(a)]{Serre} by applying the subsequent Lemma~\ref{generate}(2).

For any simple reflection $s\in W$, we say that~$L(\lambda)$ is $s$-free, resp. $s$-finite, if $s=s_\alpha$ and~$\fg_{-\alpha}$ acts freely, resp. locally finitely, on~$L(\lambda)$. In case~$\alpha$, or $\alpha/2$, is simple in~$\Delta^+$ we find that~$L(\lambda)$ is $s$-free if and only if $s\cdot\lambda\ge \lambda$ (which here is equivalent to~$s\circ\lambda\ge \lambda$), see e.g. \cite[Lemma~2.1]{CM1}. The full subcategory of~$\cO$ of modules which are finite dimensional is denoted by $\cF$, these are the {\em finite dimensional weight modules}. 

We will rely heavily on the following result of Gorelik in \cite[Theorem 1.3.1]{Gorelik}, which builds further on earlier results of Penkov and Serganova in \cite{Penkov, PS}.
\begin{prop}[M. Gorelik]\label{propGor}
Consider~$\chi$ a strongly typical central character. The category~$\fg\mbox{-mod}_\chi$ is equivalent to~$\fg_{\oa}\mbox{-mod}_{\chi^0}$ for a certain central character~$\chi^0:Z(\fg_{\oa})\to\mC$. The equivalence is given by the mutually inverse functors
$$(\res-)_{\chi^0}\qquad\mbox{and}\qquad (\ind -)_{\chi}.$$
This hence also restricts to an equivalence $\cO(\fg,\fb_{\oa})^\Lambda_{\chi}\cong\cO(\fg_{\oa},\fb_{\oa})^\Lambda_{\chi^0}$.
\end{prop}

\subsection{}\label{introHC}
We set $\UU=U\otimes U^{\op}$, so $U$-mod-$U\cong \UU$-mod. We define three injective algebra morphisms $U\hookrightarrow \UU$ by
$$\phi_l(X)=X\otimes 1,\quad \phi_r(X)=-1\otimes X\;\mbox{ and}\quad \phi_{\ad}(X)=X\otimes 1-1\otimes X$$
for all $X\in\fg$. Consequently, $\phi_r(u)=1\otimes t(u)$ for all $u\in U$.
 
We consider the functor $(-)^{\ad}$ as the restriction functor from $\UU$-mod to $U$-Mod via $\phi_{\ad}$. The category $\cH$ of {\em Harish-Chandra} bimodules is the full subcategory of~$U$-mod-$U$ of bimodules $M$ such that every indecomposable summand in~$M^{\ad}$ is in~$\cF$. Note that this is equivalent to the definition in \cite[Section~5.1]{Serre}. For a central character~$\chi$, the full subcategory of~$\cH$ of objects which are annihilated by~$\fm_\chi$ through $\phi_r$ is denoted by~$\cH_{\chi}^1$.

We consider~$\Hom_{\mC}(-,-)$ as an exact functor from~$\fg$-mod$\times\fg$-mod to $U$-mod-$U$, where for~$M,N\in\fg$-mod, $\alpha\in \Hom_{\mC}(M,N)$ and~$u,v\in U$ the map $u\alpha v\in\Hom_{\mC}(M,N)$ is defined as $(u\alpha v)(a)=u(\alpha(va))$ for~$a\in M$. The functor $\cL(-,-)$ from~$\fg$-mod$\times\fg$-mod to $\cH$ is the composition of~$\Hom_{\mC}(-,-)$ with the functor taking the maximal $\UU$-submodule which is in~$\cH$. By construction, $\cL(-,-)$ is left exact.

For any $M\in \fg$-mod, we define~${}_\times M$ and~$M_{\times}$ as objects in~$U\mbox{-mod-}U$. For $_{\times}M$ this is by keeping the left $U$-action and setting the right $U$-action to zero. For $M_{\times}$ the left action is set to zero while for any $v\in M$ and~$u\in U$ we set $vu=t(u)v$. In particular, this means 
$({}_\times M)^{\ad}\cong M\cong(M_\times)^{\ad}.$
Thus, for any $V\in \cF$ it follows that~${}_\times V$ and~$ V_\times$ belong to~$ \cH$.

The regular bimodule $U$ is in $\cH$. It is a standard property that for $V\in \cF$ and~$M\in \UU$-mod
\begin{equation}\label{EU}
\Hom_{\UU}({}_\times V\otimes U, M)\cong \Hom_{U}(V,M^{\ad}),
\end{equation}
see e.g. \cite[Section~5]{BG}. Note that tensor products of $\UU$-modules are, as usual, defined through the canonical Hopf superalgebra structure on universal enveloping algebras. For the particular example of ${}_\times V\otimes U$ above, for $v\in V$ and $u\in U$ and $X,Y\in\fg$, the action is
$$X(v\otimes u)Y=Xv\otimes uY\;+\;(-1)^{|X||v|}v\otimes XuY,$$
with $|X|$ and $|v|$ the parity of $X$ and $v$. It follows also from the definitions that
\begin{equation}\label{invariants}\Hom_{\fg}(\mC,\cL(M,N)^{\ad})\;\cong\; \Hom_{\fg}(M,N).\end{equation}

\subsection{}\label{IntroTw}
For any simple reflection $s\in W$, the endofunctor $T_s$ of $\cO(\fg,\fb_{\oa})$, called the Arkhipov {\em twisting functor}, was introduced in \cite[Section~5]{CM1}, see also~\cite{CMW, AS}. By \cite[Lemma~5.7]{CM1} we have for $\lambda\in\Lambda_0$ and $s=s_\alpha$ such that $\alpha$ or $\alpha/2$ is simple in $\Delta^+$ that
\begin{equation}
\label{twistVer}T_s M(\lambda)\;\cong\; M(s\cdot\lambda)\qquad\mbox{ if } s\cdot\lambda\le \lambda.
\end{equation}
The functor $T_s$ is right exact, see \cite[Lemma~5.4]{CM1}. In \cite[Proposition~5.11]{CM1} it is proved that $T_s$ restricts to endofunctors of each subcategory~$\cO_\chi$ and that the left derived functor $\cL T_s$ provides an auto-equivalence of the bounded derived category $\cD^b(\cO_\chi)$. Furthermore, by \cite[Lemma~5.4]{CM1}, for $M\in \cO$ we have
\begin{equation}\label{eqLT}
\begin{cases}
T_s M=0&\mbox{if $M$ is $s$-finite};\\
 \cL_1 T_s M=0& \mbox{if $M$ is $s$-free};\\
 \cL_k T_sM=0& \mbox{if $k>1$}.
\end{cases}
\end{equation}
Finally we note that, for $T^0_s$ the corresponding twisting functor on $\cO(\fg_{\oa},\fb_{\oa})$, we have
\begin{equation}\label{eqTRI}
T_s\circ\ind\,\cong\,\ind\circ T^0_s\quad\mbox{ and }\quad T_s^0\circ\res\,\cong\,\res\circ T_s.
\end{equation}

\subsection{}\label{introId}By \cite{Musson} or \cite[Theorem 15.2.4]{bookMusson}, every primitive ideal in $U(\fg)$ is in the set
$$\mathscr{X}\;=\;\{J(\lambda):=\Ann_{U}L(\lambda)\;|\quad\lambda\in\fh^\ast\},$$
for an arbitrary $\fb$. As the annihilator ideal of a module and the parity reversed module are identical, we can ignore parity. The main aim of this paper is to describe the partial order on $\mathscr{X}$ given by the inclusion order.

The poset $\mathscr{X}$ does not depend on the choice of $\fb$, but of course the labelling of ideals by $\lambda\in\fh^\ast$ does. If we consider several Borel algebras with all the same underlying $\fb_{\oa}$, then also the set $\{L(\lambda)\;|\;\lambda\in\fh^\ast\}$ does not depend on $\fb$, but still the labelling does.

\subsection{}\label{KLorder}
We define the {\em Kazhdan-Lusztig order} $\KL$ on~$\Lambda_0$, as in \cite[Definition~5.3]{CouMus}, to be the partial quasi-order transitively generated by setting $\nu\KL\lambda$ whenever~$\Ext^1_{\cO}(L(\lambda),L(\nu))\not=0$ and there is a simple reflection~$s\in W$ such that~$L(\lambda)$ is $s$-finite, whereas $L(\nu)$ is $s$-free.

We define the {\em completed Kazhdan-Lusztig order} $\KL^c$ on~$\fh^\ast$, as the partial quasi-order transitively generated by $\nu\KL^c\lambda$ when~$L(\lambda)$ is a subquotient of $T_sL(\nu)$ for some simple reflection~$s$.

It would be more correct to define both orders on the set of simple highest weight modules, rather then on their labelling set $\Lambda_0$ (or~$\fh^\ast$), see \ref{introId}. Both the $\Ext^1$-relation and the action of twisting functors are properties of $\cO(\fg,\fb_{\oa})$, so of the modules, and do not depend on~$\fb$. However we stick to the our convention, which is notationally more convenient. When talking about the (completed) KL order on~$\Lambda_0$, one hence silently assumes the choice of a Borel subalgebra. Our main result will be that $\KL^c$ describes the order on $\mathscr{X}$.

By \cite[Theorem~5.12]{CM1} we have~$\mu\KL \lambda \Rightarrow \mu\KL^c\lambda$ for $\mu,\lambda\in\Lambda_0$. We will prove in Section~\ref{secAS} that the two quasi-orders are identical if $\cO$ admits certain Kazhdan-Lusztig theories in the sense of~\cite{CPS}, which is the case for~$\fg=\mathfrak{gl}(m|n)$. For~$\fg$ a reductive Lie algebra, the (completed) KL order correspond to the left KL order of~\cite{KL}, see \cite[Theorem~5.1]{CouMus}.


\section{Some useful tools}\label{secTools}
In this section we fix an arbitrary~$\fg$ in \eqref{list} with some triangular decomposition.

\begin{lemma}\label{LemGor}
Set $\Phi: \cO(\fg,\fb_{\oa})^\Lambda_\chi\;\to\; \cO(\fg_{\oa},\fb_{\oa})^\Lambda_{\chi^0}$, the equivalence of Proposition \ref{propGor}. For $\lambda\in \Lambda_\chi$ dominant, there is $\lambda'\in \Lambda$ which is $\circ$-dominant such that for all $x\in W_\Lambda$:
$$\Phi(L(x\cdot\lambda))\cong L_0(x\circ \lambda'),\quad \Phi(\Delta(x\cdot\lambda))\cong \Delta_0(x\circ \lambda'),\quad \Phi(P(x\cdot\lambda))\cong P_0(x\circ \lambda').$$
\end{lemma}
\begin{proof}
By the BGG Theorem, \cite[Theorem~5.1]{Humphreys} and BGG reciprocity \cite[Theorem~3.11]{Humphreys}, there is a unique Verma module $\Delta_0(\lambda')$ which is projective in $\cO^\Lambda_{\chi^0}$. The equivalence in Proposition \ref{propGor} interchanges Verma modules, see \cite[3.5.1 and 2.7.2]{Gorelik}. Hence there can only be one projective Verma module in $\cO(\fg_{\oa},\fb_{\oa})^\Lambda_{\chi^0}$ and, in particular, $\Delta(\lambda)\cong\Phi^{-1}(\Delta_0(\lambda'))$.
Equation \eqref{eqTRI} and \cite[Lemma~5.4]{CM1} for the typical case, then imply by iteration that
$$\Phi(\Delta(x\cdot\lambda))\cong \Delta_0(x\cdot\lambda')\qquad \mbox{for all }x\in W_\Lambda.$$
The corresponding statement for simple and projective modules follows from taking the top and the projective cover of the Verma modules.
\end{proof}

\begin{lemma}\label{generate}
Consider any $\Lambda\in\fh^\ast/\Lambda_0$. 
\begin{enumerate}
\item The set $\Lambda$ always contains super dominant weights.
\item Fix a super dominant indecomposable projective object $P$ in~$\cO^{\Lambda}$. Every projective module in~$\cO^\Lambda$ is a direct summand of a module $V\otimes P$ for some~$V\in \cF$.
\end{enumerate}
\end{lemma}
\begin{proof}
Lemma~\ref{lemZar} implies that, for any $\mu\in\Lambda$, there exist a regular strongly typical $\nu\in\Lambda$ such that $L(\nu-\mu)$ is finite dimensional and 
such that the $P(\mu)$ is a direct summand of~$L\otimes P(\nu)$ for $L=L(\nu-\mu)^\ast\in\cF$. In particular, this already implies that any $\Lambda$ contains regular strongly typical weights. A super dominant one can then be obtained from the Weyl group action, proving part~(1).

To prove part~(2), by the above paragraph, it suffices to prove that for every strongly typical $\nu\in\Lambda$, the module $P(\nu)$ is a direct summand of~$V\otimes P(\lambda)$ for some~$V\in \cF$ and an arbitrary but fixed super dominant $\lambda$. For a specific $P(\nu)$, we use the equivalence of Proposition~\ref{propGor} (using Lemma~\ref{LemGor})  which maps $P(\nu)$ to some~$P_0(\nu')$. Similarly, we denote the image of~$P(\lambda)$ in~$\cO(\fg_{\oa},\fb_{\oa})_{\chi^0_\lambda}$ by $P_0(\lambda')$ for $\lambda'$ regular dominant for $\fg_{\oa}$.  By \cite[Theorem~3.3]{BG} there is a finite dimensional~$\fg_{\oa}$-module~$V_0$ such that~$P_0(\nu')$ is a direct summand of~$P_0(\lambda')\otimes V_0$ with $\lambda'$ regular dominant for $\fg_{\oa}$. We claim that~$P(\nu)$ is a direct summand of~$P(\lambda)\otimes \ind V_0$. Indeed, by Proposition~\ref{propGor}, the latter is equivalent to whether~$P_0(\nu')$ is a direct summand of~$\res(P(\lambda)\otimes \ind V_0)\cong \res (P(\lambda))\otimes V_0\otimes \Lambda\fg_{\ob}$. By Proposition~\ref{propGor}, $\res (P(\lambda))$ contains $P_0(\lambda')$ as a direct summand, whereas the trivial module is a direct summand of~$\Lambda\fg_{\ob}$, so $P(\nu)$ is indeed a direct summand of the tensor product. 
\end{proof}

\subsection{}
By a projective functor on (a subcategory) of~$\Zf$ we mean a direct summand of a functor $( V\otimes -)$ with~$V\in \cF$. From Proposition~\ref{propGor} we can easily obtain the following generalisation of the celebrated result of Bern\v{s}tein and Gel'fand in \cite[Theorem~3.3]{BG}. 
\begin{prop}\label{BGclass}
Fix a regular strongly typical central character~$\chi$ and consider the set $\{\lambda_1,\lambda_2,\cdots,\lambda_d\}$ of all dominant $\lambda$ with~$\chi_\lambda=\chi$. We also set $\Lambda_i=\Lambda_{\lambda_i}$.

For every $i\in\{1,\cdots,d\}$ and~$x\in W_{\Lambda_i}$ there is an exact indecomposable endofunctor $\theta^i_x$ of~$\fg\mbox{-mod}_\chi$, which preserves $\cO^{\Lambda_i}_{\chi}$, and for which $\theta^i_x\Delta(\lambda_i)\cong P(x\cdot\lambda_i)$ and~$\theta^i_x\Delta(\lambda_j)=0$ for $i\not=j$. The indecomposable projective functors on~$\fg\mbox{-mod}_\chi$ are precisely given by these functors. Furthermore, every functor $\theta^i_x$ is direct summand of a functor $\theta^i_{s_1}\theta^i_{s_2}\cdots\theta^i_{s_k}$ for some reflections $s_1,\cdots,s_k$ which are simple in~$W_{\Lambda_i}$.
\end{prop}
\begin{proof}We use the equivalences in Proposition~\ref{propGor} and define $\lambda^0_i\in \fh^\ast$ by
$$L_0(\lambda_i^0)\;\cong\;\res(L(\lambda_i))_{\chi^0}.$$

Consider a projective functor $F$ on~$\fg\mbox{-mod}_\chi$. This must be a direct summand of the functor $F_V:=V\otimes -$ on $\Zf$ for some $V\in \cF$.  Now define
$$F':=\left(\res\circ F\circ (\ind -)_{\chi}\right)_{\chi^0}:\fg_{\oa}\mbox{-mod}_{\chi^0}\to\fg_{\oa}\mbox{-mod}_{\chi^0}.$$
The functor $F'$ is clearly a direct summand of the functor 
$$\res\circ F_V\circ \ind\;\cong\; (\Lambda\fg_{\ob}\otimes \res V)\otimes -$$
and hence $F'$ is a projective functor on $\fg_{\oa}\mbox{-mod}_{\chi^0}$. Consequently, we find that every projective functor $F$ on~$\fg\mbox{-mod}_\chi$ is of the form 
$$F=\left(\ind \circ F'\circ \res(-)_{\chi^0}\right)_{\chi}$$
for some projective functor $F'$ on~$\fg_{\oa}\mbox{-mod}_{\chi^0}$.

The classification of indecomposable projective functors on~$\fg_{\oa}\mbox{-mod}_{\chi^0}$ from~\cite[Theorem~3.3(ii)]{BG}, see also \cite[Theorem~10.8]{Humphreys}, then allows us to introduce the endofunctors $\theta^i_x$ of~$\fg\mbox{-mod}_\chi$ as
$$\theta^i_x\;:=\;\left(\ind \circ F_{\xi}\circ \res(-)_{\chi^0}\right)_{\chi}\quad\mbox{for}\quad \xi=(x\circ\lambda_i^0,\lambda_i^0),$$
where we used the notation introduced in \cite[Section 3.3]{BG}. The fact that any $\theta_{x}^{i}$ is a direct summand of a composition as $\theta^i_{s_1}\theta^i_{s_2}\cdots\theta^i_{s_k}$, follows by the above construction from the corresponding claim for Lie algebras, see e.g. \cite[Corollary~5.6]{AlgCat}. One could for instance take $s_1s_2\cdots s_k$ to be a (reduced) expression for $x$ inside $W_{\Lambda_i}$.

Also the expression for $\theta_x^i\Delta(\lambda_j)$ is inherited by the corresponding properties (a) and (b) of $F_\xi$ in \cite[Theorem 3.3(ii)]{BG}, by construction and Lemma~\ref{LemGor}.

It remains to prove that every $\theta^i_x$ is indeed a projective functor on~$\fg\mbox{-mod}_\chi$. But if some~$\theta^i_x$ would not be a projective functor, then tensoring $\Delta(\lambda_i)$ with finite dimensional modules would not generate every projective module in $\cO^{\Lambda_i}$, yielding a contradiction with Lemma~\ref{generate}(2).
\end{proof}

\begin{prop}\label{propcompare}
Consider two functors $F_1,F_2$ on~$\cO(\fg,\fb)$, which 
\begin{enumerate}
\item both functorially commute with projective functors, up to isomorphism;
\item are both left (resp. right) exact;
\item are isomorphic when restricted to strongly typical regular blocks;
\end{enumerate}
then $F_1\cong F_2$.
\end{prop}
\begin{proof}
We consider~$F_i$ to be right exact, the proof for left exact functors follows from replacing projective modules by injective modules below; or simply by composing with the exact contravariant functor $\dd$. We fix an arbitrary $\cO^\Lambda$ and consider the restriction of $F_i$ as functors $\cO^\Lambda\to \cO$. By condition (3) and Lemma~\ref{generate}(1), $F_1$ and~$F_2$ are isomorphic when restricted to the full subcategory of~$\cO^\Lambda$ with one object, up to isomorphism, given by some super dominant projective module in~$\cO^\Lambda$. All projective modules in~$\cO$ are direct summands of tensor products of this module, by Lemma~\ref{generate}(2). Property (1) then implies that~$F_1$ and~$F_2$ are isomorphic when restricted to the full subcategory of projective modules. The claim hence follows from property (2).
\end{proof}

\subsection{}\label{LLL} From the PBW theorem it follows that a module for $U=U(\fg)$ is locally finite if and only if its restriction to $U_{\oa}=U(\fg_{\oa}) $ is locally finite. This justifies using the same notation $\cL(-,-)$, {\it cfr.} \ref{introHC}, for the related functors
$$U_{\oa} \mbox{-mod}\,\times\, U\mbox{-mod}\;\to\; U\mbox{-mod-}U_{\oa} \quad\mbox{and }\quad U_{\oa} \mbox{-mod}\,\times\, U_{\oa} \mbox{-mod}\;\to\; U_{\oa} \mbox{-mod-}U_{\oa} .$$
Both functors, as well as the original one in \ref{introHC}, can be defined as taking the maximal sub-vectorspace for which the adjoint $U_{\oa} $-action is locally finite. This space then naturally has the appropriate bimodule structure. We also introduce the induction functor
$$\ind^r=-\otimes_{U_{\oa}}U:\quad A\mbox{\rm-mod-}U_{\oa} \to A\mbox{-mod-}U,$$
where $A$ is either $U$ or $U_{\oa}$, and its two-sided adjoint $\res^r$. Similarly we define $\ind^l$ and $\res^l$ for the left action on bimodules.
\begin{lemma}\label{LemLLL}
For the functors $\cL$ as introduced above we have isomorphisms
\begin{enumerate}
\item $\cL(\ind -,-)\;\cong\; \ind^r \circ \cL(-,-)$ as functors $U_{\oa} \mbox{\rm-mod}\,\times\, U\mbox{\rm-mod}\;\to\; U\mbox{\rm-mod-}U$;
\item $\cL(\res -,-)\;\cong\; \res^r \circ \cL(-,-)$ as functors $U\mbox{\rm-mod}\,\times\, U_{\oa}\mbox{\rm-mod}\;\to\; U_{\oa}\mbox{\rm-mod-}U_{\oa}$;
\item $\left(\cL(\ind-,-)\right)^{\ad}\;\cong \;\ind \circ \left(\cL(-,\res-)\right)^{\ad}$ as functors $U_{\oa} \mbox{\rm-mod}\,\times\, U\mbox{\rm-mod}\;\to\; U\mbox{\rm-Mod}$.
\end{enumerate}

\end{lemma}
\begin{proof}
The first two properties follow immediately from the first observation in \ref{LLL} and the fact that induction and coinduction from $U_{\oa}$ to $U$ coincide, see \ref{Prel111}. By (1) and the analogue of (2) for $\res^l$, property (3) reduces to
$$\left(\ind^r-\right)^{\ad}\;\cong\; \ind\circ \left(\res^l-\right)^{\ad}$$
as functors from the category of $U\times U_{\oa}$-bimodules with locally finite adjoint action, to the category $U\mbox{\rm-Mod}$.

This is just a slightly alternative version of what is sometimes referred to as the ``tensor identity'' or ``projection formula''. To keep the formulae comprehensive we omit all restriction functors in the rest of the proof. We will construct, for any ad-finite $X\in U\mbox{-mod-}U_{\oa}$, a (unique) isomorphism
$$\eta: \;\, U\otimes_{U_{\oa}}(X)^{\ad}\,\,\to\,\,(X\otimes_{U_{\oa}}U)^{\ad} \qquad\mbox{with}\quad \eta(1\otimes x)= x\otimes 1\quad\mbox{for all $x\in X$,}$$
which is natural with respect to $X$. Consider therefore an arbitrary $M\in U$-mod, from the locally finite adjoint action it will actually follow that it suffices to consider $M$ to be finite dimensional. Then by using equation \eqref{EU} and applying Frobenius reciprocity twice,
$$\Hom_U(M,U\otimes_{U_{\oa}}(X)^{\ad})\;\cong\;\Hom_{U_{\oa}-U_{\oa}}({}_{\times}M\otimes U_{\oa},X)\;\cong\; \Hom_{U-U_{\oa}}(U_{\otimes U_{\oa}}({}_{\times}M\otimes U_{\oa}),X).$$
On the other hand we calculate similarly
$$\Hom_U(M, (X\otimes_{U_{\oa}}U)^{\ad})\;\cong\;\Hom_{\UU}({}_{\times}M\otimes U, X\otimes_{U_{\oa}}U)\;\cong\;\Hom_{U-U_{\oa}}({}_{\times}M\otimes U, X).$$
Now it follows from the ordinary tensor identity \cite[Proposition 6.5]{Knapp} that there is a unique isomorphism
$$U_{\otimes U_{\oa}}({}_{\times}M\otimes U_{\oa})\cong ({}_{\times}M\otimes \left(U\otimes_{U_{\oa}}U_{\oa}\right))\cong {}_{\times}M\otimes U$$
as $U\times U_{\oa}$-bimodules. By the above this induces a unique isomorphism $U\otimes_{U_{\oa}}(X)^{\ad}\,\to\,(X\otimes_{U_{\oa}}U)^{\ad}$. Tracing what happens to $1\otimes x$ is done as in the second part of the proof of \cite[Proposition 6.5]{Knapp}.
\end{proof}


\section{Harish-Chandra bimodules}\label{secHC}
In this section we fix a~$\fg$ in \eqref{list} and an arbitrary~$\Lambda\in\fh^\ast/P_0$. We choose some super dominant~$\lambda\in\Lambda$, which exists by Lemma~\ref{generate}(1), and set $\chi=\chi_\lambda$. We will study $\cH^1_\chi$, where it will be essential that $\chi$ is thus regular and strongly typical.

 The following theorem is essentially a special case of~\cite[Theorem~5.1(c)]{Serre}. 
Most of the proof uses rather standard arguments, see e.g. \cite[Section 5]{BG}, \cite[Section 10]{Joseph3}, \cite[Theorem~3.1]{MS} or \cite[Kapitel~6]{Jantzen}.

\begin{thm}[V. Mazorchuk, V. Miemietz]\label{thmMM}
We have an equivalence of categories
$\cO^{\Lambda}\cong \cH^1_{\chi}$. The mutually inverse functors are given by $\cL(\Delta(\lambda),-)$ and~$-\otimes_{ U}\Delta(\lambda).$
\end{thm}
\begin{proof}
We consider~$J=U\fm$, with~$\fm=\ker \chi$. From equation~\eqref{EU} and the fact
$$\Hom_{\UU}(U,M)\cong \Hom_{\UU}(U/J,M)$$
for any $M\in \cH^1_\chi$, it follows easily that every projective object in~$\cH^1_\chi$ is a direct summand of a bimodule of the form ${}_{\times}V\otimes U/J$ with~$V$ some projective object in~$\cF$.

We consider the functor
$$T:=-\otimes_{ U}\Delta(\lambda) \;:\;\; \cH^1_{\chi}\to \cO^{\Lambda},$$
which is right exact by construction. We clearly have
$$T({}_{\times}V\otimes U/J)=V\otimes \Delta(\lambda),$$
so in particular, $T$ maps projective objects in~$\cH^1_{\chi}$ to projective objects in~$\cO^\Lambda$. We claim that every projective object in~$\cO^\Lambda$ is a direct summand of an object in the image of~$T$. This follows from the analogue of Lemma~\ref{generate}(2), with the extra condition that the finite dimensional module $V$ is projective in $\cF$. This generalisation of Lemma~\ref{generate}(2) follows automatically as every module in~$\cF$ is the quotient of a projective object.

The last line in the proof of~\cite[Theorem~9.5]{Gorelika} implies we have an isomorphism of bimodules
$$U/J\;\;\tilde\to\;  \;\cL(\Delta(\lambda),\Delta(\lambda)).$$
Note that strictly speaking, as stated in \cite{Gorelika}, the results do not cover the exceptional basic classical Lie superalgebras. However, the arguments go through for any basic classical Lie superalgebra. This isomorphism implies in particular an isomorphism between $(U/J)^{\ad}$ and the largest locally finite submodule of~$\End_\mC(\Delta(\lambda))^{\ad}$. By equation~\eqref{EU}, this shows
$$\Hom_{\UU}({}_{\times}V \otimes U/J, U/J )\;\cong\;\Hom_U(V,\Delta(\lambda)^\ast\otimes\Delta(\lambda))\;\cong\; \Hom_U(V\otimes \Delta(\lambda),\Delta(\lambda)).$$
This can then be used to prove that
$$\Hom_{\UU}(P,P' )\cong \Hom_U(TP,TP'),$$
for all projective objects~$P,P'\in\cH^1_\chi$. By construction, this isomorphism is induced by $T$. The fact that~$T$ yields an equivalence of categories now follows from application of~\cite[Proposition~5.10]{BG}.

That the exact functor $\cL(\Delta(\lambda),-): \cO^{\Lambda}\to \cH^1_{\chi}$ is an inverse to~$T$ follows precisely as in \cite[Section~6]{BG}.
\end{proof}

\begin{rem}
Note that the theorem actually implies that~${}_{\times}V\otimes U/J$ is projective in~$\cH_\chi^1$ for any $V\in \cF$, even when $V$ is not projective. Comparison with \eqref{EU} seems to suggest that all finite dimensional summands of $M^{\ad}$ must be injective (and hence also projective) in~$\cF$, which is confirmed in the following proposition.
\end{rem}

\begin{prop}
For any $M\in \cH^1_\chi$,  the~$\fg$-module $M^{\ad}$ decomposes into modules which are projective in~$\cF$.
\end{prop}
\begin{proof}
By Theorem~\ref{thmMM}, any $M\in \cH^1_\chi$ is of the form $\cL(\Delta(\lambda),N)$ for some $N\in\cO$. By Proposition~\ref{propGor} and Lemma~\ref{LemGor}, $\Delta(\lambda)$ is a direct summand of a module $\ind \Delta_0(\lambda')$. This leads to the conclusion that $M^{\ad}$ is a direct summand of $\cL(\ind \Delta_0(\lambda'),N)^{\ad}.$ By Lemma~\ref{LemLLL}(3) we find
$$\cL(\ind \Delta_0(\lambda'),N)^{\ad}\;\cong\; \ind\left( \cL(\Delta_0(\lambda'),\res N)^{\ad}\right).$$
Hence it follows that any finite dimensional direct summand of $M^{\ad}$ is a direct summand of a module induced from a finite dimensional~$\fg_{\oa}$-module, which proves the proposition.
\end{proof}

\begin{cor}\label{corMM}
\begin{enumerate}
\item The equivalence of Theorem~\ref{thmMM} yields a one to one correspondence between the annihilator ideals of simple modules in~$\cO^\Lambda$ and left annihilator ideals of simple modules in~$\cH^1_\chi$. Concretely we have
$$\Ann L =\Lann \cL(\Delta(\lambda),L),$$
for any simple object $L$ in~$\cO^\Lambda$.
\item For any $x\in W$, there is an equivalence~$\cO^\Lambda\cong \cO^{x\cdot\Lambda}$, which interchanges simple modules with identical annihilator ideals.
\end{enumerate}
\end{cor}
\begin{proof}
Set $L'=\cL(\Delta(\lambda),L)$. From construction it follows that~$\Lann L' \supset \Ann L$. By Theorem~\ref{thmMM}, we also know $L\cong L'\otimes_U \Delta(\lambda)$ which yields $\Ann(L) \supset \Lann L'$, proving claim~(1).

In the set $\Lambda'=x\cdot\Lambda$, we can choose some~$\lambda'\in W\cdot\lambda$, which is dominant and hence super dominant. Theorem~\ref{thmMM} then also yields $\cO^{\Lambda'}\cong \cH_\chi^1$. The conclusion about annihilator ideals follows immediately from applying claim~(1) to $\cO^\Lambda\cong\cH^1_\chi\cong \cO^{x\cdot\Lambda}$.
\end{proof}

\begin{rem}
The equivalence in part~(2) can also be obtained from the results in \cite[Section~5]{CM1}, as has been worked out explicitly for type~$A$ in \cite[Proposition 3.9]{CMW}. The property about annihilator ideals in part~(2) is then precisely \cite[Remark~5.13 and Lemma~5.15]{CM1}.
\end{rem}

We prove the following analogue of~\cite[Theorem~3.2]{Vogan}. Besides the generalisation to the setting of superalgebras, the relevance lies in the property that the central character for the left action on~$\cH$ must not be equal to that of the right action, or even be regular or typical.

\begin{thm}\label{thmVogan}
Consider two simple objects~$L_1$ and~$L_2$ in~$\cH_\chi^1$. We have~$\Lann L_1 \subset \Lann L_2$
if and only if there exists a $V\in \cF$ such that~$L_2$ is a subquotient of~$L_1\otimes V_\times$.
\end{thm}

\begin{prop}\label{propVogan}
For any simple object $L$ in~$\cH_\chi^1$, there exists a $V^1\in \cF$ such that 
$$U/\Lann(L)\hookrightarrow L\otimes V^1_{\times}.$$
\end{prop}
\begin{proof}
Take any $V^1\in \cF$, which is dual to a simple submodule of $L^{\ad}$. Then by construction~$M:=L\otimes V^1_{\times}\in \cH$ is such that~$M^{\ad}$ has the trivial module in its socle. Consider the regular bimodule $U\in\cH$. Mapping $1\in U$ to a non-zero vector $\alpha$ in a trivial submodule of $M^{\ad}$ induces a $\UU$-module morphism $U\to M$. From Theorem~\ref{thmMM}, we know that~$L$ is of the form $\cL(\Delta(\lambda),L')$ for~$L'$ a simple highest weight module. By construction and equation~\eqref{invariants} we find that~$\alpha\in\cL(\Delta(\lambda)\otimes V^1,L')\subset \Hom_\mC(\Delta(\lambda)\otimes V^1,L')$ is an element of~$\Hom_{\fg}(\Delta(\lambda)\otimes V^1,L')$. So in particular the image of~$\alpha$ is $L'$. Corollary \ref{corMM}(1) implies that~$\Ann(L')=\Lann(L)$. So the only $u\in U$ which act trivially on~$\alpha$ from the left are the elements in~$\Lann(L)$. It therefore follows that the kernel of the~$\UU$-module morphism $U\to L\otimes V^1_{\times}$ is precisely
$\Lann(L)$.
\end{proof}

\begin{pfVogan}
First assume that~$L_2$ is a subquotient of~$L_1\otimes V_\times$. As $$\Lann \left(L_1\otimes V_{\times}\right)=\Lann(L_1),$$ the inclusion $\Lann (L_1)\subseteq\Lann(L_2)$ follows immediately. 

Now we prove the other direction of the assertion. We set $J_i:=\Lann L_i$ for $i\in\{1,2\}$. Proposition~\ref{propVogan} implies that there exist $V^1,V^2\in \cF$ such that
$$
U/J_1\hookrightarrow L_1\otimes V_\times^1\quad\mbox{and}\quad U/J_2\hookrightarrow L_2\otimes V_\times^2.
$$
The second property implies, by adjunction and the fact that~$L_2$ is simple, that for $E=(V^2)^\ast$ there is a morphism
$$U/J_2\otimes E_\times\tto L_2.$$
Now assume that~$J_1 \subset J_2$. By the above,~$L_2$ is a quotient of~$U/J_2\otimes E_\times$, which is a quotient of~$U/J_1\otimes E_\times$, which is a submodule of~$$L_1\otimes  (V_1\otimes E)_\times.$$
The claim hence follows for~$V=V_1\otimes E$.
\qed
\end{pfVogan}


\section{Twisting and completion functors}\label{secCompletion}
In this section we fix an arbitrary~$\Lambda\in\fh^\ast/\Lambda_0$, some simple~$s\in W_\Lambda$ and a super dominant~$\lambda\in \Lambda$. We then set $\chi=\chi_\lambda$ and $\chi^0$ equal to some central character for~$\fg_{\oa}$ for which we can apply Proposition~\ref{propGor}. We will introduce and study a completion functor $G_s$ on $\cO^\Lambda$. Note that a different approach to the completion functor for superalgebras, closer to the original definition by Enright, has recently been introduced in the more general setting of Kac-Moody superalgebras in \cite{Iohara}.

\subsection{}We propose an analogue of Joseph's version of Enright's completion functor, see \cite{Joseph2}:
$$G_s:=\cL(\Delta(s\cdot\lambda),-)\otimes_U \Delta(\lambda)\,: \;\cO^\Lambda\to\cO^{\Lambda}.$$
By construction, this functor is left exact. That $G_s$ is an endofunctor on~$\cO^\Lambda$ follows from Theorem~\ref{thmMM} and the fact that~$\cL(\Delta(s\cdot\lambda),-)$ is a functor $\cO^\Lambda\to \cH^1_{\chi}$. It could be proved that~$G_s$ does not explicitly depend on the specific super dominant $\lambda\in\Lambda$, similarly to the proof sketched in \cite[Section~2.2]{Joseph2}. However, Proposition~\ref{propcompare} and Lemma~\ref{lemRes}(3) below allow to reduce the statement to \cite[Section~2.2]{Joseph2}. Finally, note that it follows immediately from the definition that $\Ann_U G_sM\supseteq \Ann_U M$ for any $M\in \cO$ and that $G_s$ preserves the generalised central character of modules.

\begin{lemma}\label{lemRes}
Denote the completion functor of~\cite{Joseph2} on~$\cO^\Lambda(\fg_{\oa},\fb_{\oa})$ by~$G^0_s$, then 
\begin{enumerate}
\item $\res\circ G_s\;\cong\; G^0_s\circ \res,$
\item $G_s\circ\ind\;\cong\;\ind\circ G^0_s$,
\item through the equivalence in Proposition~\ref{propGor}, $G_s$ restricted to $\cO_\chi$ is interchanged with $G_s^0$ restricted to $\cO_{\chi^0}$.
\end{enumerate}
\end{lemma}
\begin{proof}
We use the functors $\cL$ from \ref{LLL}. Proposition~\ref{propGor} and Lemma~\ref{LemGor} imply that for some~$\lambda'\in \Lambda_0$, which is regular dominant for~$\fg_{\oa}$, it follows that~$\ind \Delta_0(\lambda')$ decomposes as $\Delta(\lambda)\oplus N$ where~$N_{\chi}=0$. We can therefore rewrite~$G_s$, using Lemma~\ref{LemLLL}(2), as
$$G_s\;\cong\;\cL(\Delta(s\cdot\lambda),-)\otimes_U U\otimes_{U_{\oa}}\Delta_0(\lambda')\;\cong\; \cL(\res \Delta(s\cdot\lambda),-)\otimes_{U_{\oa}}\Delta_0(\lambda').$$
By Proposition~\ref{propGor} and Lemma~\ref{LemGor} it follows that~$\res \Delta(s\cdot\lambda)$ decomposes as $\Delta_0(s\circ\lambda')\oplus N'$ where~$(N')_{\chi^0}=0$. This finally yields
$$G_s\cong \cL( \Delta_0(s\circ\lambda'),-)\otimes_{U_{\oa}}\Delta_0(\lambda').$$
Claim (1) and (2) then follow from the analogues of Lemma~\ref{LemLLL}(1) and (2) for the left action.

Claim (3) is an immediate consequence of the description of the equivalence in Proposition~\ref{propGor} and the above.
\end{proof}

\begin{thm}\label{ThmWork}
Consider~$\lambda_1,\lambda_2\in \Lambda$. We have~$J(\lambda_1)\subset J(\lambda_2)$ if and only if $L(\lambda_2)$ is a subquotient of~$G_{s_1}G_{s_2}\cdots G_{s_k} L(\lambda_1)$, for some simple reflections $s_1,s_2,\cdots,s_k\in W_\Lambda$.
\end{thm}
\begin{proof}
Set $L_1=\cL(\Delta(\lambda),L(\lambda_1))$ and~$L_2=\cL(\Delta(\lambda),L(\lambda_2))$, so~$\Lann L_1=J(\lambda_1)$ and~$\Lann L_2=J(\lambda_2)$ by Corollary \ref{corMM}(1). Theorem~\ref{thmVogan} thus implies that~$J(\lambda_1)\subset J(\lambda_2)$ if and only if there is a $V\in \cF$ such that 
$L_2$ is a subquotient of~$L_1\otimes V_{\times}$. 

The claim is therefore reduced to the claim that there exists~$V\in \cF$ such that
$$L_2\cong\cL(\Delta(\lambda),L(\lambda_2))\quad\mbox{is a subquotient of}\quad \cL(V\otimes\Delta(\lambda),L(\lambda_1))$$
if and only if
$$L(\lambda_2)\quad\mbox{is a subquotient of} \quad G_{s_1}G_{s_2}\cdots G_{s_k} L(\lambda_1),$$
for some simple reflections $s_1,s_2,\cdots,s_k\in W$.

To prove the `if' part it suffices to consider one simple reflection~$s_1$, the full statement follows by iteration. Assume that~$L(\lambda_2)$ is a subquotient of 
$$G_{s_1}L(\lambda_1)=\cL(\Delta(s_1\cdot\lambda),L(\lambda_1))\otimes_U \Delta(\lambda).$$
By Theorem~\ref{thmMM}, this is equivalent to the property that~$L_2$ is a subquotient of the bimodule~$\cL(\Delta(s_1\cdot\lambda),L(\lambda_1))$. Proposition~\ref{BGclass} implies that there exists some indecomposable projective functor $\theta_{s_1}$ such $P(s_1\cdot\lambda_1)\cong \theta_{s_1}\Delta(\lambda_1)$. Proposition~\ref{propGor} and Lemma~\ref{LemGor} allow to obtain the Verma filtration of $P(s_1\cdot\lambda)$ from the $\fg_{\oa}$-case, see e.g. \cite[2.9, Chapter 12]{Jantzen}. This implies there is a short exact sequence
$$0\to \Delta(\lambda)\to \theta_{s_1} \Delta(\lambda)\to \Delta(s_1\cdot\lambda)\to 0,$$
leading to the exact sequence
$$0\to \cL(\Delta(s_1\cdot\lambda),L(\lambda_1))\to \cL(\theta_{s_1}\Delta(\lambda),L(\lambda_1))\to L_1.$$
This shows that~$L_2$ is a subquotient of $\cL(\theta_{s_1}\Delta(\lambda),L(\lambda_1))$ which must itself be a subquotient of some~$\cL(V\otimes\Delta(\lambda),L(\lambda_1))$, which is what needed to be proved.

Now we prove the `only if' part. By Proposition~\ref{BGclass} every indecomposable projective functor on~$\fg\mbox{-mod}_{\chi}$ is a direct summand of some composition of certain~$\theta_{s_i}$. We can therefore work iteratively as above to prove the implication.
\end{proof}

\subsection{}\label{ssimple}Now we prove that~$G_s$ on~$\cO$ is right adjoint to $T_s$ of \ref{IntroTw}. For Lie algebras, this was proved in \cite[Theorem~3]{Khomenko}. Furthermore, we also prove that~$G_s\cong \dd T_s \dd$, as a generalisation of \cite[Theorem~4.1]{AS}.
Note that~$s\in W_\Lambda$ might not be simple in~$W$, so $T_s$ is defined by a reduced expression of $s$ in the Coxeter group $W$, as in \cite[Lemma~5.3]{CM1}. 
\begin{thm}\label{TsGs}
Consider the functors $T_s$ and~$G_s$ on~$\cO^\Lambda$. The functor $G_s$ is right adjoint to~$T_s$ and~$G_s\cong \dd T_s \dd$.
\end{thm}
\begin{proof}
The right exact functor $T_s$ admits a right adjoint which is a left exact endofunctor of~$\cO^\Lambda$, as argued in \cite[Section~4]{AS}. We denote this functor by~$G_s'$. The functors $T_s$ and $G_s'$ commute with projective functors by \cite[Lemma~5.9]{CM1}. The functors $T_s$ and $G_s'$ restrict to endofunctors of $\cO_{\chi'}$ for any central character $\chi'$ see \ref{IntroTw}. By equation \eqref{eqTRI} and Proposition~\ref{propGor}, for a strongly typical regular block, the functors $T_s$ and~$\dd G_s'\dd$ are the image of the corresponding functors for $\fg_{\oa}$ under the equivalence of categories. On such a block they are hence equivalent by \cite[Theorem~4.1]{AS}. The fact that~$G_s'\cong \dd T_s\dd$ on arbitrary blocks then follows from Proposition~\ref{propcompare}.

By construction, the left exact functors $G'_s$ and~$G_s$ commute with projective functors. Now on a strongly typical regular block $G_s$ and~$G_s'$ are isomorphic, by Lemma~\ref{lemRes}(3) and \cite[Theorem~3]{Khomenko}. Their equivalence hence again follows from Proposition~\ref{propcompare}. 
\end{proof}

\begin{cor}
The completion functors $G_t$ on $\cO^{\Lambda_0}$ for all simple reflections $t\in W$ satisfy the braid relations of $\cB_W$, up to isomorphism.
\end{cor}
\begin{proof}
This is a consequence of \cite[Lemma~5.3]{CM1} and Theorem~\ref{TsGs}.
\end{proof}

\begin{cor}\label{CorNew}
Consider modules $M,N\in \cO$ which are both $s$-free, then
\begin{enumerate}
\item $\Hom_{\cO}(T_s M, \dd N)\;\cong\; \Hom_{\cO}(T_s N,\dd M);$
\item $\Ext^k_{\cO}(T_s M, T_s N)\;\cong\; \Ext^k_{\cO}(M,N)$ for all $k\in\mN$.
\end{enumerate}
\end{cor}
\begin{proof}
We work in the bounded derived category $\cD^b(\cO)$. By equation \eqref{eqLT} we have
$$\Hom_{\cO}(T_s M, \dd N)\;\cong\; \Hom_{\cD^b(\cO)}(\cL T_s M,\dd N).$$
As $\cL T_s$ is an auto-equivalence of $\cD^b(\cO)$, see \ref{IntroTw}, by Theorem \ref{TsGs} we have
$$\Hom_{\cD^b(\cO)}(\cL T_s M,\dd N)\;\cong\; \Hom_{\cD^b(\cO)}( M,\dd\cL T_s N).$$
Hence part (1) follows from applying equation \eqref{eqLT} again.

Part (2) follows similarly from the properties in \ref{IntroTw}.
\end{proof}

\subsection{} Khomenko and Mazorchuk proved in \cite[Theorems~8 and~10]{Khomenko} that the completion functor, resp. twisting functor, can be realised as a {\em partial approximation}, resp. {\em partial coapproximation}, functor. For any category~$\cO_\chi^\Lambda$ and a subset $\Upsilon$ of $\Lambda^\chi$ we can mimic the procedure in \cite[Section~2.5]{Khomenko}. We define functors 
$$\mathfrak{b}_\Upsilon:\;\,\cO_\chi^\Lambda\,\to\,\cO_\chi^\Lambda\quad\mbox{and}\quad\mathfrak{d}_\Upsilon:\;\,\cO_\chi^\Lambda\,\to\,\cO_\chi^\Lambda,$$
where~$\mathfrak{b}_\Upsilon$ takes the quotient modulo the maximal submodule which does not contain simple subquotients of type~$\Upsilon$. The~$\Upsilon$-partial approximation functor~$\mathfrak{d}_\Upsilon$
is defined as first taking the maximal coextension with simple subquotients not of type~$\Upsilon$, followed by the functor~$\mathfrak{b}_{\Upsilon}$. The first map in the procedure to define~$\mathfrak{d}_{\Upsilon}$ is not a functor, {\it viz.} not well-defined on morphisms. On a module~$M$, this first map is explicitly defined as the submodule of its injective hull $I_M$ defined by taking the kernel of all morphisms $\alpha$ from $I_M$ to injective hulls of simple modules of type~$\Upsilon$ such that~$\alpha\circ\imath=0$ for $\imath: M\hookrightarrow I_M$. Note that all modules in~$\cO$ have finite length. The arguments in \cite[Section~2.5]{Khomenko} then imply that the composition~$\mathfrak{d}_{\Upsilon}$ is a well-defined functor. The dual construction gives its left adjoint functor $\tilde{\mathfrak{d}}_{\Upsilon}$, which is called the~$\Upsilon$-partial coapproximation functor.
\begin{prop}\label{propKM}
Set
$$\Upsilon=\{\mu\in \Lambda^\chi\,|\, L(\mu)\mbox{ is $s$-free}\},$$
we have isomorphisms of functors $T_s\cong \tilde{\mathfrak{d}}_{\Upsilon}$ and~$G_s\cong \mathfrak{d}_{\Upsilon}$ on~$\cO^\Lambda_\chi$.
\end{prop}
\begin{proof}
By construction, the partial (co)approximation functors commute with projective functors, see \cite[Lemma~12]{Khomenko}. Furthermore~$\tilde{\mathfrak{d}}_{\Upsilon}$ is right exact and~$\mathfrak{d}_{\Upsilon}$ left exact by \cite[Corollary~2]{Khomenko}. By proposition~\ref{propcompare} it hence suffices to prove the statements on a strongly typical regular block. 

For this we apply the equivalence of categories in Proposition~\ref{propGor}. This equivalence preserves the property of simple modules to be $s$-free or $s$-finite. Under the equivalence, the partial (co)approximation functors on $\cO_\chi(\fg,\fb_{\oa})_{\chi}$ are hence mapped to the corresponding partial (co)approximation functors on $\cO(\fg_{\oa},\fb_{\oa})_{\chi^0}$.

That $G_s$ is isomorphic to~$\mathfrak{d}_\Upsilon$ follows therefore from Lemma~\ref{lemRes}(3) and \cite[Theorem~8]{Khomenko}. The result for twisting can similarly be derived from equation \eqref{eqTRI} and the Lie algebra case in \cite[Theorem~10]{Khomenko}, or alternatively from Theorem \ref{TsGs} and the fact that $\widetilde{\mathfrak{d}}_\Upsilon$ is left adjoint to~$\mathfrak{d}_{\Upsilon}$, which can be proved identically as in \cite[Lemma~9]{Khomenko}.
\end{proof}

\subsection{} We could also have considered completion functors for arbitrary simple reflections $t$ in~$W$, rather than in~$W_\Lambda$. In case~$t\not\in W_\Lambda$, this would yield an equivalence~$\cO^\Lambda\;\tilde\to\; \cO^{t\cdot\Lambda}$, which sends simple modules to simple modules with the same annihilator ideal. We do not need to consider this, as this would just give Corollary \ref{corMM}(2) and moreover the corresponding concepts have already been worked out for twisting functors in \cite{CMW, CM1}.


\section{The Jantzen middle}\label{secAS}
In this section we will restrict to $\Lambda_0$, although the results could be extended to arbitrary~$\Lambda$. For~$\mathfrak{gl}(m|n)$, this is also immediate from the results in \cite{CMW}.

\subsection{}For~$\fg$ in \eqref{list}, a weight $\lambda\in \Lambda_0$ and a simple reflection~$s\in W$, we define the {\em Jantzen middle}~$U_s(\lambda)$ for~$L(\lambda)$ as the radical of~$T_sL(\lambda)$. Just as the set of simple modules, the set of Jantzen middles does not depend on $\fb$, only on $\fb_{\oa}$, but its labelling by $\Lambda_0$ does. Note that for Lie algebras, the usual definition of the Jantzen middle actually corresponds to the closely related radical of the shuffling functor of \cite{Carlin} acting on simple modules, see \cite{VoganKL}. The connection between both follows clearly from the construction in \cite[Appendix]{appendix}. For superalgebras this classical definition is not applicable.

By \cite[Theorem~5.12]{CM1} our definition of the Jantzen middle is either zero, if $L(\lambda)$ is $s$-finite, or fits into a short exact sequence
\begin{equation}
\label{Jmiddle}0\to U_s(\lambda)\to T_sL(\lambda)\to L(\lambda)\to 0,
\end{equation}
if $L(\lambda)$ is $s$-free. An alternative description of $U_s(\lambda)$ is given in the following proposition.
\begin{prop}\label{propU}
For $L(\lambda)$ $s$-free, the Jantzen middle~$U_s(\lambda)$ is the largest $s$-finite quotient of $\Rad P(\lambda)$.
\end{prop}
\begin{proof}
This is an immediate application of Proposition~\ref{propKM}.
\end{proof}

\subsection{} For~$\fg$ a semisimple Lie algebra, the statement that~$U_s(\lambda)$ is {\em semisimple} for all $s$ and for~$\lambda$ regular is equivalent to the statement that the KL conjecture of~\cite{KL} is true by \cite[Section~7]{AS}. This semisimplicity is thus very important but extremely difficult to prove.

For classical Lie superalgebras, it is an open question whether the Jantzen middle is always semisimple. In \cite[Corollary~5.14]{CM1}, the weaker claim that the top and socle of $U_s(\lambda)$ are isomorphic is proved. We start the investigation into the semisimplicilty of Jantzen middles for Lie superalgebras by proving the following connection with abstract KL theories and will use this to prove that the Jantzen middle is semisimple for~$\mathfrak{gl}(m|n)$.
\begin{thm}\label{thmAS}
Consider the highest weight category~$\cO^{\Lambda_0}_\chi(\fg,\fb)$ with notation as in Subsection \ref{secO}. Assume that~$\cO^{\Lambda_0}_\chi(\fg,\fb)$ admits a Kazhdan-Lusztig theory in the sense of~\cite[Definition~3.3]{CPS}. For any reflection~$s=s_\alpha\in W$ such that~$\alpha$ or $\alpha/2$ is simple in~$\Delta^+$
the Jantzen middle~$U_s(\lambda)$ is semisimple for any $\lambda\in\Lambda_0^\chi$.
\end{thm}

We start the proof with the following lemma, for which we return to arbitrary category~$\cO$, {\it i.e.} without assumptions of KL theories.
\begin{lemma}\label{lemAS}
Consider a reflection~$s\in W$ such that~$s=s_\alpha$ for~$\alpha$ or $\alpha/2$ simple in~$\Delta^+$.
Choose~$\lambda\in\Lambda_0$ such that~$L(\lambda)$ is $s$-free ($\lambda\le s\cdot\lambda$) and some~$\mu\in\Lambda_0$ with~$\mu\not=s\cdot\lambda$.
\begin{enumerate}
\item If $L(\mu)$ is $s$-finite~$(\mu>s\cdot\mu)$, then there are inclusions
\begin{eqnarray*}
\Hom_{\cO}(U_s(\lambda),\nabla(\mu))&\hookrightarrow &\Ext^1_{\cO}(\Delta(\mu),L(\lambda));\\
\Hom_{\cO}(\Delta(\mu), U_s(\lambda))&\hookrightarrow& \Ext^1_{\cO}(\Delta(\mu),L(\lambda)).
\end{eqnarray*}
\item If $L(\mu)$ is $s$-free~$(\mu\le s\cdot\mu)$, then 
$$\Hom_{\cO}(U_s(\lambda),\nabla(\mu))=0=\Hom_{\cO}(\Delta(\mu),U_s(\lambda)).$$
\end{enumerate}
\end{lemma}
\begin{proof}
As $U_s(\lambda)$ is $s$-finite by Proposition~\ref{propU}, claim~(2) follows immediately. So we consider~$\mu >s\cdot\mu$. The functor $\Hom_{\cO}(-,\nabla(\mu))$ acting on the short exact sequence
\eqref{Jmiddle}
yields a long exact sequence which contains
\begin{equation}\label{eqLem}\Hom_{\cO}(T_s L(\lambda),\nabla(\mu))\to \Hom_{\cO}(U_s(\lambda),\nabla(\mu))\to \Ext^1_{\cO}(\Delta(\mu),L(\lambda)).\end{equation}
Corollary \ref{CorNew}(1) and equation \eqref{twistVer} imply that
$$\Hom_{\cO}(T_s L(\lambda),\nabla(\mu))\cong \Hom_{\cO}(T_s\Delta(\mu), L(\lambda))\cong \Hom_{\cO}(\Delta(s\cdot\mu), L(\lambda))=0.$$
The first inclusion in claim~(1) hence follows from equation~\eqref{eqLem}.

To prove the second inclusion we start from the short exact sequence
$$0\to \Delta(s\cdot\mu)\to \Delta(\mu)\to Q\to 0,$$
defining the module~$Q$. Claim (2) applied to $\Delta(s\cdot\mu)$, followed by the observation that top$Q\cong L(\mu)$, then imply that
$$\Hom_{\cO}(\Delta(\mu), U_s(\lambda))\cong \Hom_{\cO}(Q,U_s(\lambda))\cong \Hom_{\cO}(Q,T_sL(\lambda)).$$
Precisely as in \cite[Lemma~6.2]{AL}, by reducing to either the~$\mathfrak{sl}(2)$-case of \cite[Section~6.5]{AL} or the the strongly typical~$\mathfrak{osp}(1|2)$-case, one shows that there is a short exact sequence
$$0\to Q\to T_s\Delta(s\cdot\mu)\to\Delta(s\cdot\mu)\to 0.$$
Now we apply the functor $\Hom_{\cO}(-,T_sL(\lambda))$ to this short exact sequence to find
$$\Hom_{\cO}(T_s \Delta(s\cdot\mu), T_sL(\lambda))\to \Hom_{\cO}(Q,T_sL(\lambda))\to \Ext^1_{\cO}(\Delta(s\cdot\mu),T_sL(\lambda)),$$
which is exact.
By Corollary \ref{CorNew}(2) for $k=0$, the first term is zero. By equation \eqref{twistVer} and Corollary \ref{CorNew}(2) for $k=1$, we also find that the last term is isomorphic to $\Ext^1_{\cO}(\Delta(\mu),L(\lambda)),$ which proves the second inclusion of part~(1).
\end{proof}

\begin{pfAS}
We claim that if, for two $\mu_1,\mu_2\in\Lambda_0$, we have
\begin{equation}\label{homhom1}\Hom_{\cO}(U_s(\lambda),\nabla(\mu_1))\not=0\not= \Hom_{\cO}(U_s(\lambda),\nabla(\mu_2)),\end{equation}
then $\Ext^1_{\cO}(\Delta(\mu_1),L(\mu_2))=0$; and dually if
\begin{equation}\label{homhom3}\Hom_{\cO}(\Delta(\mu_1),U_s(\lambda))\not=0\not= \Hom_{\cO}(\Delta(\mu_2),U_s(\lambda)),\end{equation}
then $\Ext^1_{\cO}(L(\mu_1),\nabla(\mu_2))=0.$ The theorem then follows from \cite[Theorem~4.1]{CPS}.

By assumption, $\cO^{\Lambda_0}_\chi(\fg,\fb)$ admits an abstract Kazhdan-Lusztig theory. Hence by \cite[Definition~3.3]{CPS}, we have some length function~$\len:\Lambda_0\to\mZ$ and a parity function~$\epsilon:\Lambda_0\to\mZ$ such that~$\epsilon(\lambda)$ is the two-sided parity of $L(\lambda)$. By {\it loc. cit.} it clearly follows that~$\epsilon(\lambda)-\len(\lambda)\equiv 0 (\mbox{mod} 2)$.

Now we prove the claims in \eqref{homhom1} and \eqref{homhom3}. Assume~$\mu_1,\mu_2$ such that we have \eqref{homhom1}. Lemma~\ref{lemAS} implies that 
$$\Ext^1_{\cO}(\Delta(\mu_1),L(\lambda))\not=0\not=\Ext^1_{\cO}(\Delta(\mu_2),L(\lambda)).$$
So by \cite[Defintion~3.3]{CPS}, $\len(\mu_1)-\len(\mu_2)$ must be even and hence~$\Ext^1_{\cO}(\Delta(\mu_1),L(\mu_2))=0$.
The case \eqref{homhom3} is proved identically.\qed
\end{pfAS}

\subsection{} The {\em relevance of the semisimplicity} of the Jantzen middles, for this work, lies in the fact that it implies that the completed KL order on $\Lambda_0$ is identical to the KL order, where both are defined in Subsection \ref{KLorder}. To prove this we observe the following proposition.
\begin{prop}\label{PropU}
If the Jantzen middle~$U_s(\lambda)$ is semisimple, then
$$U_s(\lambda)\;\cong\;\bigoplus_{L(\nu) \,s \,{\rm finite}}L(\nu)^{\oplus \mu(\lambda,\nu)}$$
where~$\mu(\lambda,\nu)=\dim\Ext^1_{\cO}(L(\lambda),L(\nu)).$
\end{prop}
\begin{proof}
This follows immediately from \cite[Theorem 5.12(ii)]{CM1}.
\end{proof}
\begin{cor}\label{corOrder}
If all Jantzen middles $U_s(\kappa)$ are semisimple for $\kappa\in \Lambda_0$ and any simple reflection~$s$, then $\nu\KL\lambda$ if and only if $\nu\KL^c\lambda$ for all $\nu,\lambda\in\Lambda_0$.
\end{cor}
\begin{proof}
This follows from Proposition \ref{PropU} and the definitions in \ref{KLorder}.
\end{proof}

\begin{thm}\label{thmKLgl}
For~$\fg=\mathfrak{gl}(m|n)$ and~$\lambda\in\Lambda_0$, the module~$U_s(\lambda)$ is semisimple. Consequently, the KL and completed KL order are identical on~$\Lambda_0$.
\end{thm}
\begin{proof}
The claim is a statement about category~$\cO(\fg,\fb_{\oa})$, so we are free to choose~$\fb$. We choose the distinguished system of roots, for which every even simple root is also simple in~$\Delta^+$.
We consider the Brundan-Kazhdan-Lusztig theory formulated in \cite{Brundan} and proved to be correct in \cite{CLW, BLW}.  It is proved explicitly in \cite[Corollary 3.3]{CS} that, in this case, the BKL theory is an abstract KL theory in the sense of~\cite{CPS}. We can hence apply Theorem~\ref{thmAS} and Corollary \ref{corOrder}.
\end{proof}

We end this section with the following immediate consequence of Theorem~\ref{thmAS}, where we recall that~$T_s$ is really defined on~$\cO(\fg,\fb_{\oa})$.
\begin{prop}
Consider~$\fg$ in \eqref{list} with fixed~$\fb_{\oa}$. If for every even simple root $\alpha$, one of the Borel subalgebras~$\fb$ containing~$\fb_{\oa}$ for which $\alpha$ or $\alpha/2$ is a simple root is such that~$\cO^{\Lambda_0}(\fg,\fb)$ admits a Kazhdan-Lusztig theory, then every Jantzen middle in~$\cO^{\Lambda_0}$ is semisimple.
\end{prop}

\section{The primitive spectrum}\label{secConclusion}
\subsection{}The results in Section~\ref{secCompletion} lead to a description of the ordering on the primitive spectrum of any Lie superalgebra in \eqref{list}.
\begin{thm}\label{thmMain}
Consider~$\fg$ in \eqref{list} and a choice of triangular decomposition. 
For any $\lambda,\nu\in \fh^\ast$ we have
$$J(\nu)\subset J(\lambda)\quad\Leftrightarrow\quad \nu\KL^c\lambda.$$
\end{thm}

\subsection{} Before proving this, we consider the particular case~$\fg=\mathfrak{gl}(m|n)$, where our results are stronger, because of the study of the Janzten middle in Section~\ref{secAS}, and confirm \cite[Conjecture~5.7]{CouMus}. Similarly our methods lead to a rather self-contained proof that the inclusion order and left Kazhdan-Lusztig order coincide for reductive Lie algebras.
\begin{thm}
Consider~$\fg=\mathfrak{gl}(m|n)$. For any $\lambda,\nu\in \Lambda_0$ we have
$$J(\nu)\subset J(\lambda)\quad\Leftrightarrow\quad \nu\KL\lambda.$$
\end{thm}
\begin{proof}
This follows immediately from Theorems \ref{thmMain} and \ref{thmKLgl}.
\end{proof}

Note that for~$\mathfrak{gl}(m|n)$, when we choose the distinguished Borel subalgebra, $\KL$ is generated by the relation~$\nu\KL\lambda$ if $\mu(\lambda,\nu)\not=0$ and there is a simple reflection~$s$ such that~$s\cdot\lambda < \lambda$ and $s\cdot\nu \ge \nu$, see \cite[Definition~5.6]{CouMus}.

\begin{pfMain}
First we note that the implication~$\mu\KL^c\lambda \Rightarrow J(\mu)\subset J(\lambda) $ follows immediately from \cite[Lemma~5.15]{CM1}, although it could also be derived from the results in Theorems~\ref{ThmWork} and~\ref{TsGs}. So we are left to prove the implication~$J(\mu)\subset J(\lambda)\Rightarrow\mu\KL^c\lambda.$

We claim that~$J(\mu)\subset J(\lambda)$ implies that~$\mu\in W\cdot\Lambda_\lambda=W\circ \Lambda_\lambda$. Indeed, the observation
$$J(\mu)\cap U_{\oa}\subset J(\lambda)\cap U_{\oa}$$ implies that~$\res L(\mu)$
must contain some subquotient of the form $L_0(w\circ\lambda)$ for some~$w\in W$, see e.g. \cite[Corollary~4.2]{CM1}. This implies that~$\mu-w\circ\lambda\in \Lambda_0$, as claimed.

First assume that~$\mu\in \Lambda_\lambda$, then the statement $J(\mu)\subset J(\lambda)\Rightarrow\mu\KL^c\lambda$ follows immediately from Theorems~\ref{ThmWork} and~\ref{TsGs}.

Now assume $\mu\in w\cdot\Lambda_\lambda$ for arbitrary $w\in W$. It suffices by iteration to consider~$w$ equal to some simple~$s\in W$ for which $s\not\in W_{\Lambda_\lambda}.$ By \cite[Lemma~5.8]{CM1}, $T_sL(\mu)\cong L(\mu')$, for some~$\mu'\in\fh^\ast$ satisfying $\mu'\KL^c\mu$ and $\mu\KL^c\mu'$. By construction, or \cite[Lemma~8.3]{CM1}, we must have~$\mu' \in \Lambda_\lambda$ and by \cite[Lemma~5.15]{CM1} we have~$J(\mu')=J(\mu)$. This reduces the claim to the previous case. Note that this reduction corresponds precisely to Corollary \ref{corMM}(2). 
\qed
\end{pfMain}

\section{The generic region}\label{secgeneric}

To gain a little more insight into the structure of the KL order we will describe it explicitly in the {\em generic} region, see \cite[Definition~7.1]{CM1} for the notion of generic weights. 
We only consider~$\fg$ equal to $\mathfrak{gl}(m|n)$ or $\mathfrak{osp}(m|2n)$. We will find that, using the deformation of Weyl group orbits introduced in \cite{CM1}, the KL order there actually reduces to the KL order for~$\fg_{\oa}$, so the one the Weyl group of \cite{KL}. This is specific to the generic region, as inclusions close the the walls of the Weyl chamber are far more prevalent for $\fg$ than for $\fg_{\oa}$, see for instance \cite[Section~7.6]{CouMus}.

We fix some Borel subalgebra~$\fb_{\oa}$ of~$\fg_{\oa}$ and only consider Borel subalgebras of~$\fg$ with~$\fb_{\oa}=\fb\cap\fg_{\oa}$. To make clear that the labelling of a simple module by its highest weight depends on the specific choice of $\fb$, we use the notation $L^{(\fb)}(\lambda)$.

\subsection{} In \cite[Section~8.1]{CM1}, for any simple reflection a class of `star actions' was defined on arbitrary weights. It was proved that in the generic region this led to a (unique) action of the Weyl group. In this section we are only interested in the generic region, so we can introduce the star action more conceptually. 
\begin{prop}\label{propCM1}
Consider $\fg$ as above. There is a unique action $\ast$ of the Weyl group $W$ on the set of generic simple modules in $\cO^\Lambda$ with the property that $s\ast L^{(\fb)}(\lambda)=L^{(\fb)}(s\cdot\lambda)$ in case $s=s_\alpha$ for $\alpha$ or $\alpha/2$ simple in $\Delta^+$.
\end{prop}
\begin{proof}
Note that for a fixed Borel subalgebra $\fb_{\oa}$ of $\fg_{\oa}$ and corresponding even simple root $\alpha\in\Delta^+_{\oa}$, we can always choose a Borel subalgebra $\fb$ of $\fg$ with $\fb\cap\fg_{\oa}=\fb_{\oa}$ such that either $\alpha$ or $\alpha/2$ is simple in $\Delta^+$. This follows e.g. from \cite[Lemma 3.3 and Definition 3.5]{Vera} and the first example in Section 6 of {\it op. cit.}

The uniqueness follows immediately from the observation in the above paragraph, as for any simple reflection $s$ we can choose a Borel subalgebra to apply the stated condition. The existence follows from \cite[Definition~8.1 and Theorem~8.10]{CM1}.
\end{proof}
By the above proposition it is immediately clear that for $\fg=\mathfrak{gl}(m|n)$, this star action is just a disguised version of the usual $\rho$-shifted action of $W$ on $\Lambda$ for the distinguished Borel subalgebra. For this choice every even simple root is simple in $\Delta^+$ and the $\rho$-shifted action is the same as the $\rho_{\oa}$-shifted action. However, for $\fg=\mathfrak{osp}(m|2n)$ this action is a new concept. 

\subsection{}\label{defCM1} For a fixed $\fb$ we can introduce an action $\ast^{\fb}$ of $W$ on $\Lambda$, defined by the condition $L^{(\fb)}(s\ast^\fb\lambda)\cong s\ast L^{(\fb)}(\lambda)$. For simple reflections $s\in W$, the action $s\ast^{\fb}$ coincides with \cite[Definition~8.1]{CM1} restricted to generic weights, for an arbitrary star action map. Now the action $\ast^{\fb}$ of $W$ on $\Lambda$ is equal to the $\rho$-shifted action of Subsection~\ref{secrho} only if $\fg=\mathfrak{gl}(m|n)$ and $\fb$ is the distinguished Borel subalgebra.

\begin{thm}\label{generic}
Let $\le_L$ denote the left KL order on $W$ from \cite{KL}. 
\begin{enumerate}
\item Consider~$\fg=\mathfrak{gl}(m|n)$, with distinguished Borel subalgebra and $\lambda,\mu\in \Lambda_0$ which are dominant and generic, then
$$w_1\cdot\mu\KL w_2\cdot\lambda\quad\Leftrightarrow\quad\begin{cases}\mu=\lambda&\mbox{and}\\w_2\le_L w_1.\end{cases}$$
\item Consider~$\fg=\mathfrak{gl}(m|n)$ or~$\fg=\mathfrak{osp}(m|2n)$, with arbitrary Borel subalgebra $\fb$ and $\lambda,\mu\in \Lambda_0$ which are dominant and generic, then
$$w_1\ast^\fb \mu\KL^c w_2\ast^\fb\lambda\quad\Leftrightarrow\quad\begin{cases}\mu=\lambda&\mbox{and}\\w_2\le_L w_1.\end{cases}$$
\end{enumerate}
\end{thm}

\subsection{}Before proving this we need to recall the result of Penkov in \cite[Theorem~2.2]{Penkov}. For any $\nu\in\Lambda_0$, we set $S_\nu$ equal to the subset of $\Delta_{\ob}^+$ of odd roots $\gamma$ which satisfy $\langle\nu+\rho,\gamma\rangle\not=0$. 
For any subset $I\subset S_\nu$, we denote by $|I|\in \Lambda_0$ the sum of all the roots in the subset. Then for $\nu$ generic we have
\begin{equation}\label{eqPenkov}\res L(\nu)\;\cong\;\oplus_{I\subset S_\nu}L_0(\nu-|I|).
\end{equation}
Note also that for any simple reflection $s\in W$ we clearly have
\begin{equation}\label{eqSs}
S_{s\cdot\nu}\;=\; (s(S_\nu)\cap \Delta_{\ob}^+)\;\,\cup \;\,(-s(S_\nu)\cap \Delta_{\ob}^+).
\end{equation}

\begin{lemma}\label{LemRes}
Consider~$\fg=\mathfrak{gl}(m|n)$ or~$\fg=\mathfrak{osp}(m|2n)$, with arbitrary Borel subalgebra $\fb$ and $\mu\in \Lambda_0$ dominant and generic. For $\kappa\in\Lambda_0$ dominant and any $w,w'\in W$ we have
$$[\res L^{(\fb)}(w\ast^{\fb}\mu): L_0(w'\circ \kappa)]\;=\; \delta_{w,w'} [\res L^{(\fb)}(\mu): L_0(\kappa)].$$
Hence there is a set of $\circ$-regular dominant weights $R_\mu$ such that for all $w\in W$,
$$\res L^{(\fb)}(w\ast^{\fb} \mu)=\bigoplus_{\kappa\in R_\mu}L_0(w\circ \kappa).$$
\end{lemma}
\begin{proof}
By definition it suffices to prove this for one specific Borel subalgebra for each $\fg$. For $\fg=\mathfrak{gl}(m|n)$ we consider the distinguished Borel subalgebra, meaning that $\ast^{\fb}$ is the $\rho$-shifted action. As $\Delta_{\ob}^+$ is preserved by the Weyl group action, equation~\eqref{eqSs} implies $S_{w\cdot\mu}=w(S_{\mu})$. Equation~\eqref{eqPenkov} therefore shows that for any $w\in W$
$$\res L(w\cdot\mu)\;\cong\; \oplus_{I\subset S_\mu}L_0(w\cdot(\mu-|I|)).$$
The claim hence follows from the fact that the $\rho$-shifted action equals the $\rho_{\oa}$-shifted action.

For $\fg=\mathfrak{osp}(m|2n)$ we consider the Borel subalgebra and notation for the corresponding simple roots of \cite[Section 8.2]{CM1}. Consider some generic integral $\nu\in\fh^\ast$. Then, by Proposition~\ref{propCM1}, for any simple reflection $s$ different from $s_{2\delta_n}$ we have $s\circ\nu=s\cdot\nu=s\ast^{\fb}\nu$. We also have $S_{s\cdot\nu}=s(S_\nu)$ for such $s$ by equation~\eqref{eqSs}, as $s$ preserves $\Delta^+_{\ob}$. Therefore equation~\eqref{eqPenkov}
implies that 
\begin{equation}\label{eq1}\res L(s\ast^\fb\nu)\;\cong\; \bigoplus_{I\subset S_\nu}L_0(s\circ (\nu-|I|)).\end{equation}
Now consider $s=s_{2\delta_n}$. We will prove that equation~\eqref{eq1} still holds. Based on equation~\eqref{eqSs} we introduce $D_\nu:=-s(S_\nu)\cap \Delta_{\ob}^+$ which, for $s=s_{2\delta_n}$, is given by
$$D_\nu=\{-s(\gamma)\,|\,\gamma\in\Delta^+_{\ob}\,\mbox{ with $\langle \nu+\rho,\gamma\rangle\not=0$ and $\langle \gamma, \delta_n\rangle\not=0$}\}.$$
Based on this we consider two cases (A) and (B).

(A) Assume that there is no $\gamma\in \Delta^+_{\ob}$ with $\langle \nu+\rho,\gamma\rangle=0$ and $\langle \gamma, \delta_n\rangle\not=0$. Then, by \cite[Lemma 8.12]{CM1}, we have $s\ast^\fb\nu=s\cdot\nu$ and by \eqref{eqPenkov}
\begin{equation}\label{eq2}\res L(s\ast^\fb\nu)\;\cong\; \bigoplus_{I\subset S_{s\cdot\nu}}L_0(s\cdot\nu-|I|).\end{equation}
By equation \eqref{eqSs} we have
$$\{|I|\,|\, I\subset S_{s\cdot\nu}\}=\{|I|+\Sigma\,|\, I\subset s(S_{\nu})\}=\{s(|I|)+\Sigma\,|\, I\subset S_{\nu}\},$$
with $\Sigma=|D_\nu|$. By assumption we have
$D_\nu=\{\gamma\in \Delta_{\ob}^+\,|\,\langle\gamma,\delta_n\rangle\not=0\}$ and it then follows easily that $\Sigma=\rho_{\ob}-s(\rho_{\ob})$.
This means that equation~\eqref{eq2} can indeed be rewritten as equation~\eqref{eq1}. 

(B) Assume there is a $\gamma\in \Delta^+_{\ob}$ with $\langle \lambda+\rho,\gamma\rangle=0$ and $\langle \gamma, \delta_n\rangle\not=0$. This $\gamma$ must be unique, see e.g. \cite[Lemma 7.7]{CM1}. Then $s\ast^\fb \nu=s\cdot\nu+s(\gamma)$ by \cite[Lemma 8.12]{CM1}, so by \cite[Lemma~7.8]{CM1} we have $S_{s\ast^\fb \nu}=S_{s\cdot\nu}$ and then by \eqref{eqPenkov} we find
\begin{equation}\label{eq3}\res L(s\ast^\fb\nu)\;\cong\; \bigoplus_{I\subset S_{s\cdot\nu}}L_0(s\cdot\lambda+s(\gamma)-|I|).\end{equation}
Similarly as in (A) we can calculate that now $|D_\nu|=\rho_{\ob}-s(\rho_{\ob})+s(\gamma)$, showing that also equation~\eqref{eq3} can be rewritten in the form of \eqref{eq1}.

In conclusion, equation~\eqref{eq1} holds for arbitrary simple reflections, from which the claim follows by iteration.
\end{proof}

\begin{pfGeneric}
It suffices to prove claim (2), as claim (1) is a special case, by Proposition~\ref{propCM1} and Theorem \ref{thmKLgl}. We fix a Borel subalgebra $\fb$ and leave out reference to it in simple modules and star actions. We will work in the Grothendieck groups of category~$\cO$ for $\fg$ and $\fg_{\oa}$, for which the images of the simple modules constitute a basis.

Take $w\in W$ and a simple reflection $s\in W$. As twisting functors commute with translation functors, see \cite[Theorem~3.2]{AS}, there is a set $E_w^s$ of elements of the Weyl group such that for each integral $\circ$-regular dominant $\kappa\in\fh^\ast$ we have
$$[T_s^0 L_0(w\circ\kappa)]=\bigoplus_{x\in E^s_w}[L_0(x\circ\kappa)].$$
Lemma~\ref{LemRes} and equation \eqref{eqTRI} then imply that for $\mu$ integral dominant and generic
$$[T_s L(w\ast\mu)]=\bigoplus_{x\in E^s_w}[L(x\ast\mu)],$$
for the same set $E_w^s$. This concludes the proof.
 \qed

\end{pfGeneric}

\subsection*{Acknowledgement}
The author wishes to thank Volodymyr Mazorchuk and Catharina Stroppel for stimulating discussions, Ian Musson for discussions during a previous attempt to prove the conjecture and Maria Gorelik and Ian Musson for useful comments on the first version of the manuscript. The research was supported by Australian Research Council Discover-Project Grant DP140103239 and an FWO postdoctoral grant.

\begin{flushleft}
	K.~Coulembier\qquad \url{kevin.coulembier@sydney.edu.au}
	
	School of Mathematics and Statistics, University of Sydney, NSW 2006, Australia

\end{flushleft}

\end{document}